\newif\ifLTEX
\LTEXtrue

\documentclass[a4paper]{amsart}

\usepackage{mathtools,amssymb}
\usepackage[abbrev]{amsrefs}
\usepackage{xparse}

\ifLTEX
\usepackage{graphicx}
\else
\usepackage[dvipdfmx]{graphicx}
\fi
 
\usepackage{hyperref}
\usepackage{subcaption}
\usepackage{enumitem}
\usepackage{amsthm}
\usepackage{pdfpages}
\usepackage[all]{xy} 
\usepackage{float}
\usepackage{marginfix}
\usepackage{bm}

\newtheorem{thm}{Theorem}[section]
\newtheorem{lem}[thm]{Lemma}

\newtheorem{cor}[thm]{Corollary}
\newtheorem{prop}[thm]{Proposition}

\theoremstyle{definition}

\newcommand{\R}{\mathbb{R}}

\newcommand{\Z}{\mathbb{Z}}

\newcommand{\M}{\mathrm{Mod}_{0,n+1}}
\newcommand{\Mb}{\mathrm{Mod}_{0,n}^1}

\newcommand{\PM}{\mathrm{PMod}_{0,n+1}}

\newcommand{\HQ}{H_1(QB_n; \Z )}

\numberwithin{equation}{section}

\allowdisplaybreaks
\title[Minimal generating sets for quasitoric braid group]{Minimal generating sets and the abelianization for the quasitoric braid group}

\author[G.~Omori]{Genki Omori}
\address{
(Genki Omori)
Department of Mathematics, Faculty of Science and Technology, Tokyo University of Science, 2641 Yamazaki, Noda-shi, Chiba, 278-8510 Japan
}
\email{omori\_genki@rs.tus.ac.jp}


\date{\today}

\begin{document}
\maketitle
\begin{abstract}
A toric braid is a braid whose closure is a torus link in $\R ^3$. 
Manturov~\cite{Manturov} generalized toric braids that is called \emph{quasitoric braids} and showed that the subset of quasitoric braids in the classical braid group is a subgroup of the braid group. 
We call this subgroup the \emph{quasitoric braid group}. 
In this paper, we give two minimal generating sets for the quasitoric braid group and determine its abelianization. 
The minimalities of these two generating sets are obtained from a lower bound by the number of generators for the abelianization.  
\end{abstract}

\section{Introduction}

Let $B_{n}$ be the classical braid group of $n\geq 1$ strands and we call an element in $B_n$ a \emph{$n$-braid} or a \emph{braid}. 
Since $B_1$ is trivial and $B_2$ is isomorphic to $\Z $, we assume that $n\geq 3$ in this paper. 
A \emph{link} is a smoothly embedded disjoint circles in $\R ^3$.  
Alexander~\cite{Alexander} proved that every link is represented by a closure of a braid (see also~\cite{Vogel}). 
A \emph{torus link} is a link which is included in a standardly embedded torus in $\R ^3$. 
A braid $\beta $ is a \emph{toric braid} if the closure of $\beta $ is a torus link. 
By the definition, a torus link is a closure of a braid $\beta (n,m)=(\sigma _1\sigma _2 \cdots \sigma _{n-1})^m\in B_n$ for some positive integers $n$ and $m$, where $\sigma _i$ is a half-twist braid as on the right-hand side in Figure~\ref{fig_braid_product_def} (for precise definitions, see Section~\ref{section_braid-quasitoric}). 
We call the braid $\beta (n,m)$ a \emph{$(n,m)$-toric braid}. 
Manturov~\cite{Manturov} introduced a generalization $(\sigma _1^{\varepsilon _1^1}\sigma _2^{\varepsilon _2^1}\cdots \sigma _{n-1}^{\varepsilon _{n-1}^1})(\sigma _1^{\varepsilon _1^2}\sigma _2^{\varepsilon _2^2}\cdots \sigma _{n-1}^{\varepsilon _{n-1}^2})\cdots (\sigma _1^{\varepsilon _1^m}\sigma _2^{\varepsilon _2^m}\cdots \sigma _{n-1}^{\varepsilon _{n-1}^m})$ of the $(n,m)$-toric braid $\beta (n,m)$, where $\varepsilon _i^j\in \{ \pm 1\}$. 
Such a braid is called a \emph{$(n,m)$-quasitoric braid} or a \emph{$n$-quasitoric braid}. 
Lamm~\cite{Lamm1, Lamm2} independently introduced a generalization of toric braids which are called \emph{rosette braids} and are conjugate to a quasitoric braids in the braid group. 
Manturov~\cite{Manturov} and Lamm~\cite{Lamm1, Lamm2} independently proved that every link is represented by a closure of a quasitoric braid or a rosette braid, respectively. 
Moreover, by the proof, we see that every closure of a $n$-braid is represented by a closure of a $n$-quasitoric braid.  

We denote by $QB_n$ a subset of $B_n$ which consists of $n$-quasitoric braids. 
By Lemma~1 in \cite{Manturov}, we see that the subset $QB_n$ is a subgroup of $B_n$, and call the subgroup $QB_n$ the \emph{quasitoric braid group} of $n$ strands. 
Lamm~\cite{Lamm1, Lamm2} independently proved that the set of rosette braids is also a subgroup of $B_n$. 
Let $\Psi \colon B_n \to S_n$ be the surjective homomorphism which is defined by $\Psi (\sigma _i)=(i\ i+1)$ for $i\in \{ 1, 2, \dots , n\}$, where $S_n$ is the symmetric group of degree $n$. 
Then we have $B_n=\Psi ^{-1}(S_n)$, and $B_n$ is generated by two elements by a well-known fact. 
Since $B_n$ is not a cyclic group for $n\geq 3$, this generating set is minimal. 
The \emph{pure braid group} is the kernel of $\Psi $, namely, $PB_n=\Psi ^{-1}(\{ 1\} )$. 
By Corollary~\ref{cor_mini-gen_PB}, $PB_n$ has $\binom{n}{2}$ generators and this generating set is minimal. 

In generally, there is a natural problem that for a given subgroup $H$ of $S_n$, what is the minimal number of generators of $\Psi ^{-1}(H)$. 
We denote 
$\rho =\begin{pmatrix}
1 & 2 & \cdots & n-1 & n \\
2 & 3 & \cdots & n & 1 \\
\end{pmatrix}
\in S_n$ and consider the order $n$ cyclic subgroup $\left< \rho \right> \cong \Z _n=\Z /n\Z $ of $S_n$. 
The image of a $n$-quasitoric braid with respect to $\Psi $ lies in $\left< \rho \right>$, moreover, by Lemma~4 in~\cite{Manturov}, $PB_n$ is included in $QB_n$. 
Hence $\Psi ^{-1}(\Z _n[\rho ])=QB_n$.  
In this paper, we answer the problem above for the case that $H=\Z _n[\rho ]$ and the main theorem in this paper is as follows. 

\begin{thm}\label{thm_mini-gen}
\begin{enumerate}
\item For $n\geq 3$ is odd, $QB_n$ is generated by $\frac{n+1}{2}$ elements. 
\item For $n\geq 4$ is even, $QB_n$ is generated by $\frac{n+2}{2}$ elements. 
\end{enumerate}
\end{thm}

Theorem~\ref{thm_mini-gen} is proved in Section~\ref{section_mini-gen} and we give two explicit minimal generating sets for $QB_n$ (Theorems~\ref{thm_mini-gen1} and~\ref{thm_mini-gen1}). 

The integral first homology group $H_1(G; \Z )$ of a group $G$ is isomorphic to the abelianization of $G$. 
The integral first homology group of $QB_n$ is as follows: 

\begin{thm}\label{thm_abel}
\[
H_1(QB_n; \Z )\cong \left\{ \begin{array}{ll}
 \Z ^{\frac{n-1}{2}}\oplus \Z _n&\text{if } n\geq 3 \text{ is odd},   \\
 \Z ^{\frac{n}{2}}\oplus \Z _{\frac{n}{2}}& \text{if } n\geq 4 \text{ is even}.
 \end{array} \right.
\]
\end{thm}

Theorem~\ref{thm_abel} is proved in Section~\ref{section_abel_proof}. 
The results in Theorems~\ref{thm_mini-gen} and~\ref{thm_abel} for $n=3$ are obtained from results of Shigeta~\cite{Shigeta}. 
He gave a finite presentation for $QB_3$ which have two generators.  
For a group $G$, the minimal number of generators for $H_1(G ; \Z )$ gives a lower bound of the minimal number of generators for $G$.  
By Theorem~\ref{thm_abel}, we see that the generating set for $QB_n$ in Theorem~\ref{thm_mini-gen} is minimal. 
To determine the abelianization, we give a finite presentation for $QB_n$ in Section~\ref{section_pres_quasitoric} (Proposition~\ref{prop_pres_QB}) and one of $PB_n$ in Section~\ref{section_pres_purebraid} (Proposition~\ref{prop_pres_PB}). 
We remark that the presentation for $PB_n$ in Proposition~\ref{prop_pres_PB} is independently given by Namanya~\cite{Namanya}.  

As related studies, the author~\cite{Omori1, Omori2} gave small generating sets for some subgroups of mapping class groups of a 2-sphere with marked points. 
Since the braid group $B_n$ is isomorphic to the mapping class group of a 2-disk with $n$ marked points, braid groups are related to mapping class groups of a 2-sphere with marked points via forgetful maps (for the case of pure braid groups and pure mapping class groups, there is the exact sequence~(\ref{exact_pure})). 
The author's generating sets for these subgroups in~\cite{Omori1, Omori2} are minimal except for several cases.

\section{Preliminaries}\label{Preliminaries}

\subsection{The braid group and the quasitoric braid group}\label{section_braid-quasitoric}

In this section, we review the definitions of the braid group and the quasitoric braid group. 
We take $n$ points $p_i=(i,0,0)$ $(i=1, \dots , n)$ in $\R ^3$ and a 2-disk $D$ in $\R^2\times \{ 0\}\subset \R ^3$ whose interior includes the points $p_1, p_2, \dots , p_n$. 
Then a \emph{$n$-braid} is a smoothly embedded $n$ disjoint arcs in $D\times [0,1]\subset \R^3$ whose end points lie  in $\{ p_1, p_2, \dots , p_n\} \times \{ 0, 1\}$ such that the each arc transversely intersects with $D\times \{ t\}$ for $t\in [0,1]$ at one point. 
For $n$-braids $b_1, b_2$, we define the product $b_1b_2$ by a $n$-braid as in the center of Figure~\ref{fig_braid_product_def}, that is obtained by stacking the scaled $b_1$ on the top of the scaled $b_2$.  
The \emph{(classical) braid group} $B_n$ of $n$ strands is the group of isotopy classes of $n$-braids relative to $\{ p_1, p_2, \dots , p_n\} \times \{ 0, 1\}$ whose product is induced by the product of braids. 
We abuse notation and denote a braid and its isotopy class relative to $\{ p_1, p_2, \dots , p_n\} \times \{ 0, 1\}$ by the same symbol. 
Let $\sigma _i$ for $i\in \{ 1, 2, \dots , n-1\} $ be a $n$-braid as on the right-hand side in Figure~\ref{fig_braid_product_def}. 
As a well-known fact, $B_n$ is generated by $\sigma _1, \sigma _2 \dots , \sigma _{n-1}$. 

\begin{figure}[h]
\includegraphics[scale=0.74]{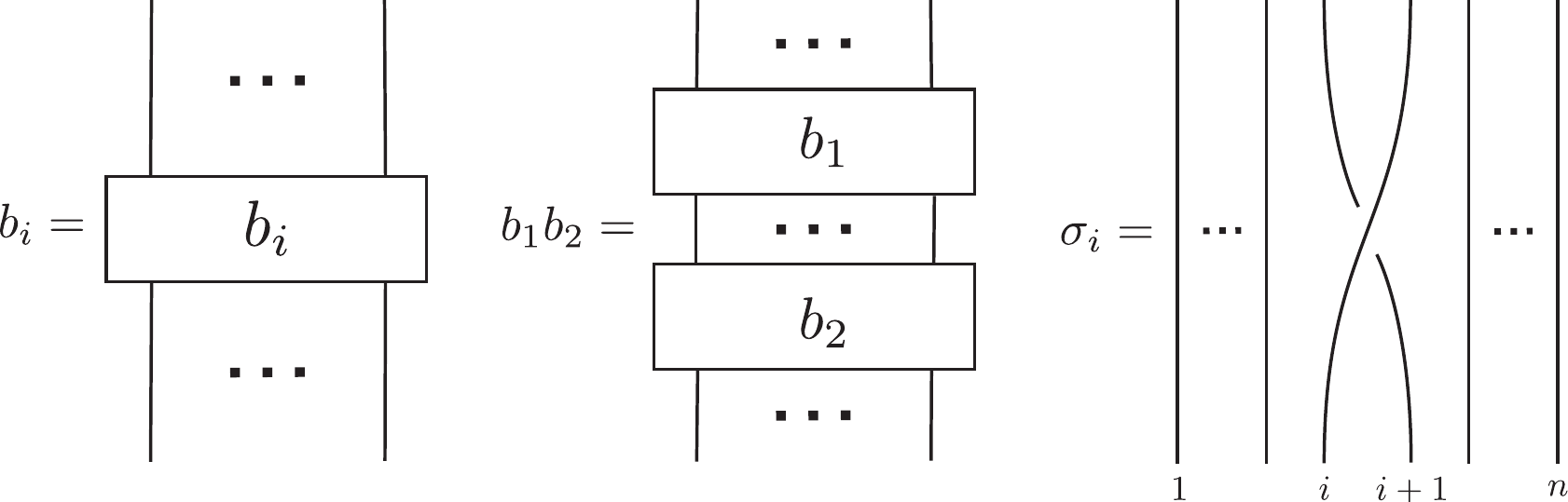}
\caption{The product $b_1b_2$ for $b_i$ $(i=1, 2)$ and the half-twist $\sigma_i\in B_{n}$ $(1\leq i\leq n-1)$.}\label{fig_braid_product_def}
\end{figure}

We denote $\beta (n,m)=(\sigma _1\sigma _2 \cdots \sigma _{n-1})^m\in B_n$ and call this $n$-braid the \emph{$(n,m)$-toric braid}. 
As on the left-hand side in Figure~\ref{fig_pq-toric_braid}, the closure of $\beta (n,m)$ is the $(n,m)$-torus link. 
Remark that the image of a toric braid by the projection to the $yz$-plane is described as a shadow as in the center of Figure~\ref{fig_pq-toric_braid}.

We call a $n$-braid which is denoted by a product
\[
(\sigma _1^{\varepsilon _1^1}\sigma _2^{\varepsilon _2^1}\cdots \sigma _{n-1}^{\varepsilon _{n-1}^1})(\sigma _1^{\varepsilon _1^2}\sigma _2^{\varepsilon _2^2}\cdots \sigma _{n-1}^{\varepsilon _{n-1}^2})\cdots (\sigma _1^{\varepsilon _1^m}\sigma _2^{\varepsilon _2^m}\cdots \sigma _{n-1}^{\varepsilon _{n-1}^m})
\] 
for some $\varepsilon _i^j\in \{ \pm 1\}$ a \emph{$(n,m)$-quasitoric braid} or a \emph{$n$-quasitoric braid}. 
Quasitoric braids are introduced by Manturov~\cite{Manturov}. 
For example, the braid as on the right-hand side in Figure~\ref{fig_pq-toric_braid} is a diagram of the $(4,3)$-quasitoric braid $(\sigma_1\sigma_2\sigma_3)(\sigma_1^{-1}\sigma_2^{-1}\sigma_3)(\sigma_1^{-1}\sigma_2\sigma_3)$. 
By the definition, a product of quasitoric braids is also a quasitoric braid. 
We remark that a $(n,m)$-quasitoric braid and $\beta(n,m)$ have the same projection to the $yz$-plane, and a diagram of a $(n,m)$-quasitoric braid is obtained from a diagram of the $(n,m)$-toric braid $\beta (n,m)$ by crossing change at some crossings. 

\begin{figure}[h]
\includegraphics[scale=0.76]{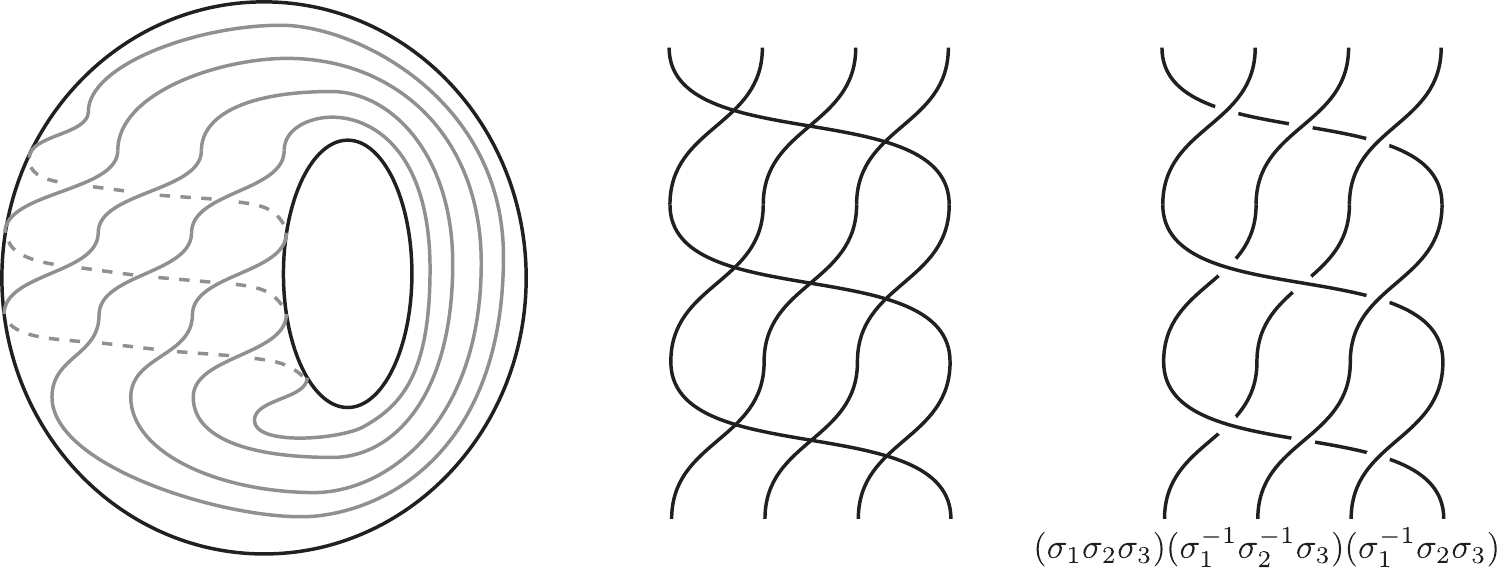}
\caption{The closure of $\beta (4,3)$, the shadow of $(4,3)$-toric braid $\beta (4,3)$, and a $(4,3)$-quasitoric braid.}\label{fig_pq-toric_braid}
\end{figure}

Let $QB_n$ be a subset of $B_n$ which consists of $n$-quasitoric braids. 
By Lemma~1 in~\cite{Manturov}, the inverse of a $n$-quasitoric braid is a product of $n$-quasitoric braids, namely, the subset $QB_n$ is a subgroup of $B_n$. 
We call $QB_n$ the \emph{quasitoric braid group} of $n$ strands. 
We regard the symmetric group $S_n$ of degree $n$ as the group of self-bijections on the set $\{ p_1, p_2, \dots , p_n \}$. 
Let 
\[
\Psi \colon B_n \to S_n
\]
be the surjective homomorphism which is defined by $\Psi (\sigma _i)=(i\ i+1)$ for $i\in \{ 1, 2, \dots , n\}$, and denote 
$\rho =\begin{pmatrix}
1 & 2 & \cdots & n-1 & n \\
2 & 3 & \cdots & n & 1 \\
\end{pmatrix}
\in S_n$. 
Since $\Psi (\sigma _1^{\varepsilon _1}\sigma _2^{\varepsilon _2}\cdots \sigma _{n-1}^{\varepsilon _{n-1}})=\rho $ for any $\varepsilon _1, \varepsilon _2, \dots , \varepsilon _{n-1}\in \{ \pm 1\}$, we have $\Psi (QB_n)=\left< \rho \right> \cong \Z _n$. 
The \emph{pure braid group} $PB_n$ is the kernel of the homomorphism $\Psi $. 
Since $PB_n$ is included in $QB_n$ by Lemma~4 in~\cite{Manturov}, we have the following exact sequence: 
\begin{eqnarray}\label{exact_quasi}
1\longrightarrow PB_n \longrightarrow QB_n \stackrel{\Psi }{\longrightarrow }\Z _n[\rho ]\longrightarrow 1. 
\end{eqnarray}

\subsection{Group extensions and presentations for groups}\label{section_presentation_exact}

To give a finite presentation of $QB_n$, for computation of its abelianization, via the exact sequence~(\ref{exact_quasi}), we review a relationship between a group extension and group presentations in this section from Section~3 in \cite{Hirose-Omori}. 
Let $G$ be a group and let $H=\bigl< X\mid R\bigr>$ and $Q=\bigl< Y\mid S\bigr>$ be presented groups which have the short exact sequence 
\[
1\longrightarrow H\stackrel{\iota }{\longrightarrow }G\stackrel{\nu }{\longrightarrow }Q\longrightarrow 1.
\]
We take a preimage $\widetilde{y}\in G$ of $y\in Q$ with respect to $\nu $ for each $y\in Q$. 
Then we put $\widetilde{X}=\{ \iota (x) \mid x\in X\} \subset G$ and $\widetilde{Y}=\{ \widetilde{y} \mid y\in Y\} \subset G$. 
Denote by $\widetilde{r}$ the word in $\widetilde{X}$ which is obtained from $r\in R$ by replacing each $x\in X$ by $\iota (x)$ and also denote by $\widetilde{s}$ the word in $\widetilde{Y}$ which is obtained from $s\in S$ by replacing each $y\in Y$ by $\widetilde{y}$. 
We note that $\widetilde{r}=1$ in $G$. 
Since $\widetilde{s}\in G$ is an element in $\ker \nu =\iota (H)$ for each $s\in S$, there exists a word $v_{s}$ in $\widetilde{X}$ such that $\widetilde{s}=v_{s}$ in $G$. 
Since $\iota (H)$ is a normal subgroup of $G$, for each $x\in X$ and $y\in Y$, there exists a word $w_{x,y}$ in $\widetilde{X}$ such that $\widetilde{y}\iota (x)\widetilde{y}^{-1}=w_{x,y}$ in $G$. 
The next lemma follows from an argument of the combinatorial group theory (for instance, see \cite[Proposition~10.2.1, p139]{Johnson}).

\begin{lem}\label{presentation_exact}
Under the situation above, the group $G$ admits the presentation with the generating set $\widetilde{X}\cup \widetilde{Y}$ and following defining relations:
\begin{enumerate}
 \item[(A)] $\widetilde{r}=1$ \quad for $r\in R$,
 \item[(B)] $\widetilde{s}=v_{s}$ \quad for $s\in S$,
 \item[(C)] $\widetilde{y}\iota (x)\widetilde{y}^{-1}=w_{x,y}$ \quad for $x\in X$ and $y\in Y$.
\end{enumerate} 
\end{lem}

\section{The abelianization for the quasitoric braid group}\label{section_abel}

The aim of this section is the calculation of the abelianization of the guasitoric braid group $QB_n$. 

\subsection{A finite presentation for the pure braid group}\label{section_pres_purebraid}

In this section, we give a finite presentation for the pure braid group $PB_n$ to obtain a finite presentation for $QB_n$ in the next section. 
Let $t_{i,j}\in B_n$ for $1\leq i<j\leq n$ be the $n$-braid which is described as the result of full-twisting the strings from $i$-th to $j$-th as in Figure~\ref{fig_full-twist-def}. 
We can see that $t_{i,j}\in PB_n$ and prove the following proposition in this section: 

\begin{figure}[h]
\includegraphics[scale=1.0]{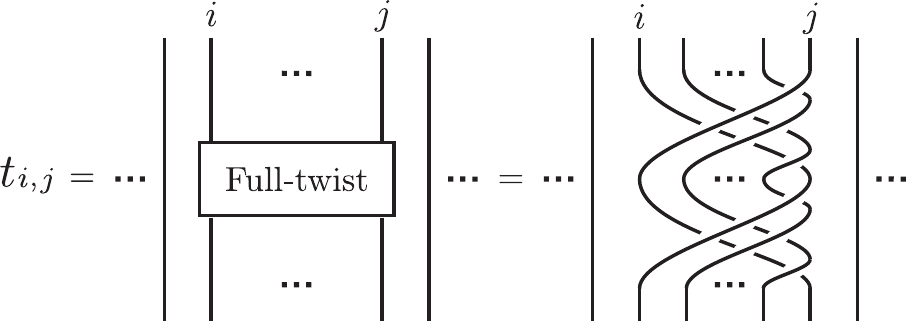}
\caption{The full-twist $t_{i,j}$ for $1\leq i<j\leq n$.}\label{fig_full-twist-def}
\end{figure}

\begin{prop}\label{prop_pres_PB}
The group $PB_n $ admits the presentation with generators $t_{i,j}$ for $1\leq i<j\leq n$ and the following defining relations: 
\begin{enumerate}
\item $t_{i,j}t_{k,l}=t_{k,l}t_{i,j}$ \quad for $j<k$, $k\leq i<j\leq l$, or $l<i$, 
\item $t_{j,m-1}^{-1}t_{k,m-1}t_{j,l-1}t_{i,k-1}t_{i,l-1}^{-1}=t_{i,l-1}^{-1}t_{i,k-1}t_{j,l-1}t_{k,m-1}t_{j,m-1}^{-1}$ \quad for $1\leq i<j<k<l<m\leq n$.\\ 
\end{enumerate}
\end{prop}

We remark that the presentation for $PB_n$ in Proposition~\ref{prop_pres_PB} is independently given by Namanya~\cite{Namanya}. 
Namanya~\cite{Namanya} gave this presentation by technical algebraic computations, on the other hand, we give the presentation above by using the presentation for the pure mapping class group of a 2-disk with marked points in Proposition~3.1 of~\cite{Hirose-Omori}. 

Let $D$ and $p_i\in D$ for $1\leq i\leq n$ be the 2-disk and the $n$ points in the interior of $D$ which are taken in Section~\ref{section_braid-quasitoric}. 
Then we take a 2-disk $D^\prime $ with a single marked point $p_{n+1}$ in the interior and regard a 2-sphere $S^2$ as the surface which is obtained from $D\sqcup D^\prime $ by gluing $\partial D^\prime $ to $\partial D$.  
The \emph{mapping class group} $\M $ of $S^2$ with marked points $p_1, \dots , p_{n+1}$ is the group of isotopy classes of orientation-preserving self-diffeomorphisms on $S^2$ fixing $\{ p_1, \dots , p_{n+1}\}$ setwise. 
We denote by $\Mb$ the group of isotopy classes of orientation-preserving self-diffeomorphisms on $D$ fixing $\{ p_1, \dots , p_{n}\}$ setwise and $\partial D$ point wise. 
The inclusion $\iota \colon D\hookrightarrow S^2$ and extensions to $S^2$ of diffeomorphisms on $D$ by the identity map induces a homomorphism 
\[
\iota _\ast \colon \Mb \to \M 
\]
whose kernel is generated by the right-handed Dehn twist $t_{\partial D}$ along $\partial D$ (we remark that this is not surjective).  
As a well-known fact, $B_n$ is isomorphic to $\Mb$ by the correspondence that $\sigma _i\in B_n$ maps to a anticlockwise half-twist which transpose $p_i$ and $p_{i+1}$ as in Figure~\ref{fig_sigma_i-new}. 
By this isomorphism, we identify $B_n$ with $\Mb $ and each element $b\in B_n$ with the image in $\Mb $ of $b$, respectively. 
Hence we have a homomorphism $\iota _\ast \colon B_n \to \M $ whose kernel is generated by $t_{1,n}$. 
 
\begin{figure}[h]
\includegraphics[scale=0.9]{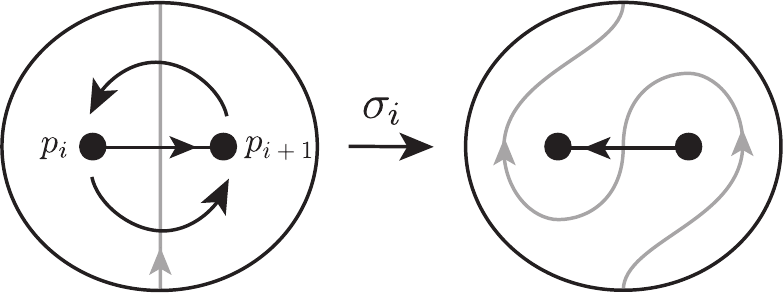}
\caption{The half-twist $\sigma _{i}$ for $1\leq i\leq n-1$.}\label{fig_sigma_i-new}
\end{figure}
 
The \emph{pure mapping class group} $\PM $ is the subgroup of $\M$ which consists of elements in $\M$ fixing $\{ p_1, \dots , p_{n+1}\}$ pointwise.  
By restricting $\iota _\ast $ to $PB_n$, we have the following exact sequence: 
\begin{eqnarray}\label{exact_pure}
1\longrightarrow \Z [t_{1,n}] \longrightarrow PB_n \stackrel{\iota _\ast }{\longrightarrow }\PM \longrightarrow 1. 
\end{eqnarray}
Let $\gamma _{i,j}$ for $1\leq i<j\leq n$ be a simple closed curve on the interior of $D\subset S^2$ as in Figure~\ref{fig_homeo_t_ij-delta}. 
We remark that the image $\iota _\ast (t_{i,j})$ for $1\leq i<j\leq n$ is the right-handed Dehn twist $t_{\gamma _{i,j}}$ along $\gamma _{i,j}$. 
We abuse notation and simply denote $t_{\gamma _{i,j}}=t_{i,j}$. 
The next proposition is given by Proposition~3.1 of~\cite{Hirose-Omori}. 

\begin{figure}[h]
\includegraphics[scale=1.0]{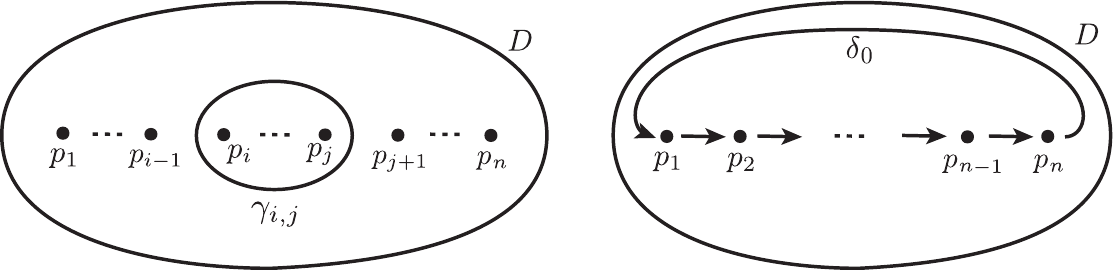}
\caption{Simple closed curve $\gamma _{i,j}$ for $1\leq i<j\leq n$ and a diffeomorphism $\delta _0$.}\label{fig_homeo_t_ij-delta}
\end{figure}

\begin{prop}\label{prop_pres_pmod}
The group $\PM $ admits the presentation with generators $t_{i,j}$ for $1\leq i<j\leq n$ with $(i,j)\not= (1,n)$, and the following defining relations: 
\begin{enumerate}
\item $t_{i,j}t_{k,l}=t_{k,l}t_{i,j}$ \quad for $j<k$, $k\leq i<j\leq l$, or $l<i$, 
\item $t_{j,m-1}^{-1}t_{k,m-1}t_{j,l-1}t_{i,k-1}t_{i,l-1}^{-1}=t_{i,l-1}^{-1}t_{i,k-1}t_{j,l-1}t_{k,m-1}t_{j,m-1}^{-1}$ \quad for $1\leq i<j<k<l<m\leq n$.\\ 
\end{enumerate}
\end{prop}

We remark that the relation~(2) in Proposition~\ref{prop_pres_pmod} is called a \emph{pentagonal relation} in~\cite{Hirose-Omori}.

\begin{proof}[The proof of Proposition~\ref{prop_pres_PB}]
We apply Lemma~\ref{presentation_exact} to the exact sequence~(\ref{exact_pure}) and the finite presentation for $\PM $ in Proposition~\ref{prop_pres_pmod}. 
Since each simple closed curve $\gamma _{i,j}$ lies in the interior of $D\subset S^2$, if $\gamma _{i,j}$ and $\gamma _{k,l}$ are disjoint in $S^2=D\cup D^\prime $, then they are also disjoint in $D$. 
Hence the relation $t_{i,j}t_{k,l}=t_{k,l}t_{i,j}$ holds in $PB_n$ for $j<k$, $1\leq k\leq i<j\leq l\leq n$, or $l<i$. 
By Lemma~2.6 in~\cite{Hirose-Omori}, the relation 
\[
t_{j,m-1}^{-1}t_{k,m-1}t_{j,l-1}t_{i,k-1}t_{i,l-1}^{-1}=t_{i,l-1}^{-1}t_{i,k-1}t_{j,l-1}t_{k,m-1}t_{j,m-1}^{-1}
\]
for $1\leq i<j<k<l<m\leq n$ holds in the mapping class group of the regular neighborhood of $\gamma _{j,m-1}\cup \gamma _{k,m-1}\cup \gamma _{j,l-1}\cup \gamma _{i,k-1}\cup \gamma _{i,l-1}$ in $D\subset S^2$. 
Thus, pentagonal relations also hold in $PB_n$. 
By applying Lemma~\ref{presentation_exact}, 
$PB_n$ has the presentation whose generators are $t_{i,j}$ for $1\leq i<j\leq n$ and the following defining relations: 
\begin{enumerate}
\item[(A)] no relations, 
\item[(B)] 
\begin{enumerate}
\item $t_{i,j}t_{k,l}=t_{k,l}t_{i,j}$ \quad for ``$j<k$, $k\leq i<j\leq l$, or $l<i$'' and $(i,j)\not=(1,n)$, 
\item $t_{j,m-1}^{-1}t_{k,m-1}t_{j,l-1}t_{i,k-1}t_{i,l-1}^{-1}=t_{i,l-1}^{-1}t_{i,k-1}t_{j,l-1}t_{k,m-1}t_{j,m-1}^{-1}$ \quad for $1\leq i<j<k<l<m\leq n$, 
\end{enumerate}
\item[(C)] $t_{i,j}t_{1,n}t_{i,j}^{-1}=t_{1,n}$ for $1\leq i<j\leq n$ with $(i,j)\not= (1,n)$. 
\end{enumerate}
The relations~(B)~(a) and (C) coincide with the relation~(1) in Proposition~\ref{prop_pres_PB} and the relation~(B)~(b) coincides with the relation~(2) in Proposition~\ref{prop_pres_PB}, respectively. 
Therefore, we have completed the proof of Proposition~\ref{prop_pres_PB}.   
\end{proof}

In the last of this section, we give remarkable two corollaries. 
The next corollary follows from easy computations of abelianizations from presentations in Propositions~\ref{prop_pres_PB} and~\ref{prop_pres_pmod}. 

\begin{cor}\label{cor_abel_PB}
For $n\geq 3$, we have 
\begin{enumerate}
\item $H_1(PB_n; \Z )\cong \Z ^{\binom{n}{2}}$,
\item $H_1(\PM ; \Z )\cong \Z ^{\binom{n}{2}-1}$.
\end{enumerate}
\end{cor}

The next corollary immediately follows from Propositions~\ref{prop_pres_PB}, \ref{prop_pres_pmod}, and~Corollary~\ref{cor_abel_PB}. 

\begin{cor}\label{cor_mini-gen_PB}
For $n\geq 3$, 
\begin{enumerate}
\item $PB_n$ is generated by $t_{i,j}$ for $1\leq i<j\leq n$ and this generating set is minimal,
\item $\PM $ is generated by $t_{i,j}$ for $1\leq i<j\leq n$ with $(i,j)\not =(1,n)$ and this generating set is minimal. 
\end{enumerate}
\end{cor}

\subsection{A finite presentation for the quasitoric braid group}\label{section_pres_quasitoric}

In this section, we give a finite presentation for the quasitoric braid group $QB_n$ to compute its abelianization in the next section. 
We denote the $(n,1)$-toric braid by
\[
\delta _0=\beta (n,1)=\sigma _1\sigma _2 \cdots \sigma _{n-1}\in B_n.
\] 
The $(n,1)$-toric braid $\delta _0$ is a braid as in Figure~\ref{fig_quasitoric_delta0} and satisfies that  
$\Psi (\delta _0)=\rho =\begin{pmatrix}
1 &  \cdots & n-1 & n \\
2 &  \cdots & n & 1 \\
\end{pmatrix}
\in S_n$. 
Under the identification $B_n\cong \Mb $ which is defined before Figure~\ref{fig_sigma_i-new}, $\delta _0$ is regarded as a diffeomorphism on $D$ which is described as the result of cyclic rotation of order $n$ as on the right-hand side in Figure~\ref{fig_homeo_t_ij-delta}.  
First, the next lemma is immediately obtained from an argument in Figure~\ref{fig_t_ij-delta_conj}. 

\begin{figure}[h]
\includegraphics[scale=1.0]{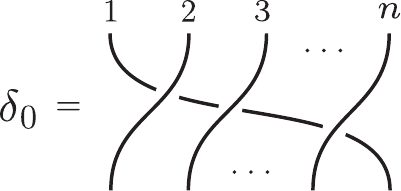}
\caption{The $(n,1)$-toric braid $\delta _0=\beta (n,1)=\sigma _1\sigma _2 \cdots \sigma _{n-1}\in B_n$.}\label{fig_quasitoric_delta0}
\end{figure}

\begin{lem}\label{lem_t_ij-delta_conj}
For $2\leq i<j\leq n$, the relation $\delta _0^{-1}t_{i,j}\delta _0=t_{i-1,j-1}$ holds in $QB_n$. 
\end{lem}

\begin{figure}[h]
\includegraphics[scale=0.8]{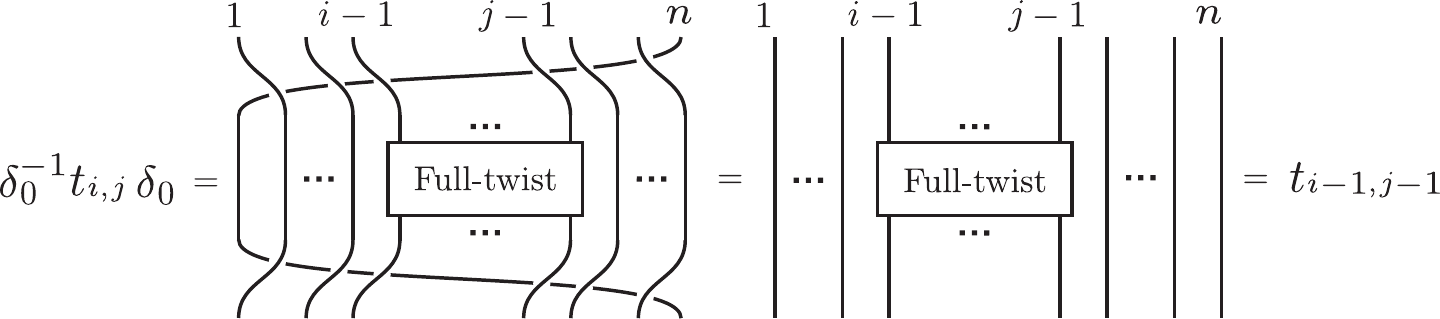}
\caption{The relation $\delta _0^{-1}t_{i,j}\delta _0=t_{i-1,j-1}$ for $2\leq i<j\leq n$.}\label{fig_t_ij-delta_conj}
\end{figure}

From here,  we regard $t_{i,i}$ for $1\leq i\leq n$ as the trivial element in $B_n$.  
We give the following proposition in this section: 

\begin{prop}\label{prop_pres_QB}
The group $QB_n $ admits the presentation with generators $\delta _0$ and $t_{i,j}$ for $1\leq i<j\leq n$, and the following defining relations: 
\begin{enumerate}
\item $t_{i,j}t_{k,l}=t_{k,l}t_{i,j}$ \quad for $j<k$, $k\leq i<j\leq l$, or $l<i$, 
\item $t_{j,m-1}^{-1}t_{k,m-1}t_{j,l-1}t_{i,k-1}t_{i,l-1}^{-1}=t_{i,l-1}^{-1}t_{i,k-1}t_{j,l-1}t_{k,m-1}t_{j,m-1}^{-1}$ \quad for $1\leq i<j<k<l<m\leq n$,
\item $\delta _0^n=t_{1,n}$,
\item $\delta _0t_{i,j}\delta _0^{-1}=
\left\{ \begin{array}{ll}
 t_{i+1,j+1} & \text{for }  j<n,   \\
 t_{1,n} & \text{for } (i,j)=(1,n), \\
 t_{2,n}^{-1}t_{1,i}^{-1}t_{2,i}t_{i+1,n}t_{1,n} & \text{for }j=n \text{ and }i>1.
 \end{array} \right. $
\end{enumerate}
\end{prop}

\begin{proof}
We apply Lemma~\ref{presentation_exact} to the exact sequence~(\ref{exact_quasi}), that is 
\[
1\longrightarrow PB_n \longrightarrow QB_n \stackrel{\Psi }{\longrightarrow }\Z _n[\rho ]\longrightarrow 1, 
\] 
and the finite presentation for $PB_n$ in Proposition~\ref{prop_pres_PB}. 
Since $\Psi (\delta _0)=\rho $ and $\delta _0^n=t_{1,n}$, the relation $\rho ^n=1$ in $\Z _n[\rho ]$ lifts to the relation $\delta _0^n=t_{1,n}$ in $QB_n$. 
Thus, by applying Lemma~\ref{presentation_exact}, 
$PB_n$ has the presentation whose generators are $\delta _0$ and $t_{i,j}$ for $1\leq i<j\leq n$ and the following defining relations: 
\begin{enumerate}
\item[(A)]  
\begin{enumerate}
\item $t_{i,j}t_{k,l}=t_{k,l}t_{i,j}$ \quad for $j<k$, $k\leq i<j\leq l$, or $l<i$, 
\item $t_{j,m-1}^{-1}t_{k,m-1}t_{j,l-1}t_{i,k-1}t_{i,l-1}^{-1}=t_{i,l-1}^{-1}t_{i,k-1}t_{j,l-1}t_{k,m-1}t_{j,m-1}^{-1}$ \quad for $1\leq i<j<k<l<m\leq n$, 
\end{enumerate}
\item[(B)] $\delta _0^n=t_{1,n}$,
\item[(C)] $\delta _0t_{i,j}\delta _0^{-1}=w_{t_{i,j},\delta _0}$ for $1\leq i<j\leq n$, 
\end{enumerate}
where $w_{t_{i,j},\delta _0}$ is some product of $t_{k,l}$ for $1\leq k<l\leq n$. 
The generators and the relations~(A)~(a), (b), and (B) of this presentation coincide with the generators and the relations~(1), (2), and (3) in Proposition~\ref{prop_pres_QB}, respectively. 
Hence, it is enough for completing the proof of Proposition~\ref{prop_pres_QB} to prove that the relation~(C) above coincides with the relation~(4) in Proposition~\ref{prop_pres_QB}.

In the case of $j<n$, we can see that $\delta _0t_{i,j}\delta _0^{-1}=t_{i+1,j+1}$ by Lemma~\ref{lem_t_ij-delta_conj}, and in the case of $(i,j)=(1,n)$, we have $\delta _0t_{1,n}\delta _0^{-1}=t_{1,n}$ since $t_{1,n}$ commutes with all elements in $B_n$. 
Hence we have $w_{t_{i,j},\delta _0}=t_{i+1,j+1}$ for $j<n$ and $w_{t_{i,j},\delta _0}=t_{1,n}$ for $(i,j)=(1,n)$, and the relations~(C) in these cases coincide with the relations~(4) in Proposition~\ref{prop_pres_QB} for the same cases.  

In the case of $j=n$ and $i>1$, the conjugation $\delta _0t_{i,j}\delta _0^{-1}=\delta _0t_{i,n}\delta _0^{-1}$ is the right-handed Dehn twist along $\delta _0(\gamma _{i,n})$ that is a simple closed curve as in Figure~\ref{fig_lantern_rel-no1}. 
By a lantern relation as in Figure~\ref{fig_lantern_rel-no1}, we have
\[
t_{1,i}t_{2,n}\cdot \delta _0t_{i,n}\delta _0^{-1}=t_{2,i}t_{i+1,n}t_{1,n}\ \Longleftrightarrow \ \delta _0t_{i,n}\delta _0^{-1}=t_{2,n}^{-1}t_{1,i}^{-1}t_{2,i}t_{i+1,n}t_{1,n}.
\]
Thus, in this case, we have $w_{t_{i,n},\delta _0}=t_{2,n}^{-1}t_{1,i}^{-1}t_{2,i}t_{i+1,n}t_{1,n}$ and the relation~(C) also coincide with the relation~(4) in Proposition~\ref{prop_pres_QB}. 
Therefore, we have completed the proof of Proposition~\ref{prop_pres_QB}.   
\end{proof}

\begin{figure}[h]
\includegraphics[scale=0.8]{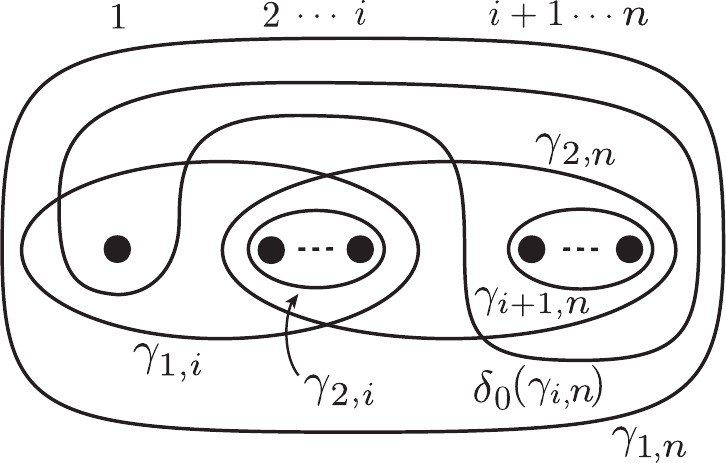}
\caption{Simple closed curves which are appeared in the lantern relation $t_{1,i}t_{2,n}\cdot \delta _0t_{i,n}\delta _0^{-1}=t_{2,i}t_{i+1,n}t_{1,n}$ for $2\leq i\leq n-1$.}\label{fig_lantern_rel-no1}
\end{figure}

\subsection{The proof of Theorem~\ref{thm_abel}}\label{section_abel_proof}

In this section, we calculate the abelianization of $QB_n$ by using the presentation in Proposition~\ref{prop_pres_QB}. 
For conveniences, we abuse notation and denote an element of $QB_n$ and its homology class in $\HQ $ by the same symbol in this section. 
First, the next lemma is immediately obtained from Lemma~\ref{lem_t_ij-delta_conj}. 

\begin{lem}\label{lem_proof_abel0}
For $1\leq i<j\leq n$, the relation $t_{i,j}=t_{1,j-i+1}$ holds in $\HQ $. 
\end{lem}

The next two lemmas give relations in $\HQ $ which are obtained from the relations~(4) in Proposition~\ref{prop_pres_QB} for $j=n$ and $i>1$. 

\begin{lem}\label{lem_proof_abel1}
When $n\geq 3$ is odd, in $H_1(QB_n;\Z )$, the relations~(4) in Proposition~\ref{prop_pres_QB} for $j=n$ and $i>1$ are equivalent to the following relations: 
\begin{enumerate}
\item $t_{1,n-i}=t_{1,i-1}^{-1}t_{1,i}t_{1,n-1}t_{1,n-i+1}\delta _0^{-n}$ \quad for $2\leq i\leq \frac{n-1}{2}$, 
\item $t_{1,n-1}^n\delta _0^{-n(n-2)}=1$. 
\end{enumerate}
\end{lem}

Since $t_{1,1}$ is trivial in $QB_n$, we remark that the relation~(1) in Lemma~\ref{lem_proof_abel1} for $i=2$ is equivalent to the relation $t_{1,n-2}=t_{1,2}t_{1,n-1}^2\delta _0^{-n}$. 

\begin{proof}
First, by the relation~(3) in Proposition~\ref{prop_pres_QB}, the relation~(4) in Proposition~\ref{prop_pres_QB} for $j=n$ and $i>1$ is equivalent to the relation $t_{i,n}=t_{2,n}^{-1}t_{1,i}^{-1}t_{2,i}t_{i+1,n}\delta _0^n$ in $\HQ $. 
By the relation in Lemma~\ref{lem_proof_abel0}, we have
\begin{eqnarray}
t_{i,n}=t_{2,n}^{-1}t_{1,i}^{-1}t_{2,i}t_{i+1,n}\delta _0^n &\Longleftrightarrow & t_{1,n-i+1}=t_{1,n-1}^{-1}t_{1,i}^{-1}t_{1,i-1}t_{1,n-i}\delta _0^n\notag \\
&\Longleftrightarrow & t_{1,n-i}=t_{1,i-1}^{-1}t_{1,i}t_{1,n-1}t_{1,n-i+1}\delta _0^{-n}.\label{rel4_1}
\end{eqnarray}
Hence the relations~(4) in Proposition~\ref{prop_pres_QB} for $j=n$ and $2\leq i\leq \frac{n-1}{2}$ are equivalent to the relation~(1) in Lemma~\ref{lem_proof_abel1}. 

In the case of $j=n$ and $\frac{n+3}{2}\leq i\leq n-1$, the integer $i^\prime =n-i+1$ satisfies that $2\leq i^\prime \leq \frac{n-1}{2}$. 
Then the relation~(\ref{rel4_1}) is equivalent to  
\begin{eqnarray*}
&&t_{1,(n-i+1)-1}=t_{1,n-(n-i+1)}^{-1}t_{1,n+1-(n-i+1)}t_{1,n-1}t_{1,n-i+1}\delta _0^{-n}\\ 
&\Longleftrightarrow & t_{1,i^\prime -1}=t_{1,n-i^\prime }^{-1}t_{1,n+1-i^\prime }t_{1,n-1}t_{1,i^\prime }\delta _0^{-n}\\
&\Longleftrightarrow & t_{1,n-i^\prime }=t_{1,i^\prime -1}^{-1}t_{1,i^\prime }t_{1,n-1}t_{1,n+1-i^\prime }\delta _0^{-n}.
\end{eqnarray*}
This relation coincides with the relation~(\ref{rel4_1}) for $2\leq i\leq \frac{n-1}{2}$, that is the relation~(1) in Lemma~\ref{lem_proof_abel1}. 

In the case of $j=n$ and $i=\frac{n+1}{2}$, the relation~(\ref{rel4_1}) is $t_{1,\frac{n-1}{2}}=t_{1,\frac{n-1}{2}}^{-1}t_{1,\frac{n+1}{2}}t_{1,n-1}t_{1,\frac{n+1}{2}}\delta _0^{-n}$ and this relation is equivalent to the relation    
\begin{eqnarray}
t_{1,\frac{n-1}{2}}^2=t_{1,\frac{n+1}{2}}^2t_{1,n-1}\delta _0^{-n}.\label{rel4_2}
\end{eqnarray}
By using the relations~(1) in Lemma~\ref{lem_proof_abel1} for $i=\frac{n-1}{2}, \frac{n-3}{2}, \dots , 2$ consecutively, we have 
\begin{eqnarray*}
t_{1,\frac{n+1}{2}}&=&t_{1,\frac{n-3}{2}}^{-1}t_{1,\frac{n-1}{2}}t_{1,n-1}\underline{t_{1,\frac{n+3}{2}}}\delta _0^{-n}\\
&=&t_{1,\frac{n-3}{2}}^{-1}t_{1,\frac{n-1}{2}}t_{1,n-1}(t_{1,\frac{n-5}{2}}^{-1}t_{1,\frac{n-3}{2}}t_{1,n-1}t_{1,\frac{n+5}{2}}\delta _0^{-n})\delta _0^{-n}\\
&=&t_{1,\frac{n-5}{2}}^{-1}t_{1,\frac{n-1}{2}}t_{1,n-1}^2\underline{t_{1,\frac{n+5}{2}}}\delta _0^{-2n}\\
&=&t_{1,\frac{n-5}{2}}^{-1}t_{1,\frac{n-1}{2}}t_{1,n-1}^2(t_{1,\frac{n-7}{2}}^{-1}t_{1,\frac{n-5}{2}}t_{1,n-1}t_{1,\frac{n+7}{2}}\delta _0^{-n})\delta _0^{-2n}\\
&=&t_{1,\frac{n-7}{2}}^{-1}t_{1,\frac{n-1}{2}}t_{1,n-1}^3\underline{t_{1,\frac{n+7}{2}}}\delta _0^{-3n}\\
&\vdots &\\
&=&t_{1,2}^{-1}t_{1,\frac{n-1}{2}}t_{1,n-1}^{\frac{n-5}{2}}\underline{t_{1,n-2}}\delta _0^{-\frac{n-5}{2}n}\\
&=&t_{1,2}^{-1}t_{1,\frac{n-1}{2}}t_{1,n-1}^{\frac{n-5}{2}}(t_{1,2}t_{1,n-1}^2\delta _0^{-n})\delta _0^{-\frac{n-5}{2}n}\\
&=&t_{1,\frac{n-1}{2}}t_{1,n-1}^{\frac{n-1}{2}}\delta _0^{-\frac{n-3}{2}n}. 
\end{eqnarray*}
Thus the relation~(\ref{rel4_2}) is equivalent to the relation
\begin{align*}
t_{1,\frac{n-1}{2}}^2&=\underline{t_{1,\frac{n+1}{2}}^2}t_{1,n-1}\delta _0^{-n}=(t_{1,\frac{n-1}{2}}t_{1,n-1}^{\frac{n-1}{2}}\delta _0^{-\frac{n-3}{2}n})^2t_{1,n-1}\delta _0^{-n}\\
&=t_{1,\frac{n-1}{2}}^2t_{1,n-1}^{n-1}\delta _0^{-(n-3)n}t_{1,n-1}\delta _0^{-n} \\
&=t_{1,\frac{n-1}{2}}^2t_{1,n-1}^{n}\delta _0^{-(n-2)n}. 
\end{align*}
Therefore, this relation is equivalent to the relation~(2) in Lemma~\ref{lem_proof_abel1} and we have completed the proof of Lemma~\ref{lem_proof_abel1}.   
\end{proof}

\begin{lem}\label{lem_proof_abel2}
When $n\geq 4$ is even, in $H_1(QB_n;\Z )$, the relations~(4) in Proposition~\ref{prop_pres_QB} for $j=n$ and $i>1$ are equivalent to the following relations: 
\begin{enumerate}
\item $t_{1,n-i}=t_{1,i-1}^{-1}t_{1,i}t_{1,n-1}t_{1,n-i+1}\delta _0^{-n}$ \quad for $2\leq i\leq \frac{n-2}{2}$, 
\item $t_{1,n-1}^{\frac{n}{2}}\delta _0^{-\frac{n(n-2)}{2}}=1$. 
\end{enumerate}
\end{lem}

\begin{proof}
By an argument similar to the top of the proof of Lemma~\ref{lem_proof_abel1}, the relations~(4) in Proposition~\ref{prop_pres_QB} for $j=n$ and $i>1$ are equivalent to the relations~(\ref{rel4_1}), that are $t_{1,n-i}=t_{1,i-1}^{-1}t_{1,i}t_{1,n-1}t_{1,n-i+1}\delta _0^{-n}$, and these relations coincide with the relations~(1) in Lemma~\ref{lem_proof_abel2} for the cases $2\leq i\leq \frac{n-2}{2}$ and $\frac{n+4}{2}\leq i\leq n-1$. 

In the case of $j=n$ and $i=\frac{n}{2}$, the relation~(\ref{rel4_1}) is $t_{1,\frac{n}{2}}=t_{1,\frac{n-2}{2}}^{-1}t_{1,\frac{n}{2}}t_{1,n-1}t_{1,\frac{n+2}{2}}\delta _0^{-n}$ and this relation is equivalent to the relation    
\begin{eqnarray}
t_{1,\frac{n-2}{2}}=t_{1,n-1}t_{1,\frac{n+2}{2}}\delta _0^{-n}.\label{rel4_3}
\end{eqnarray}
Similarly, for $j=n$ and $i=\frac{n+2}{2}$, the relation~(\ref{rel4_1}) is equivalent to the relation
\[
t_{1,\frac{n-2}{2}}=t_{1,\frac{n}{2}}^{-1}t_{1,\frac{n+2}{2}}t_{1,n-1}t_{1,\frac{n}{2}}\delta _0^{-n}\Longleftrightarrow t_{1,\frac{n-2}{2}}=t_{1,n-1}t_{1,\frac{n+2}{2}}\delta _0^{-n}
\]
and this relation coincides with the relation~(\ref{rel4_3}).  
By using the relations~(1) in Lemma~\ref{lem_proof_abel2} for $i=\frac{n-2}{2}, \frac{n-4}{2}, \dots , 2$ consecutively, we have 
\begin{eqnarray*}
t_{1,\frac{n+2}{2}}&=&t_{1,\frac{n-4}{2}}^{-1}t_{1,\frac{n-2}{2}}t_{1,n-1}\underline{t_{1,\frac{n+4}{2}}}\delta _0^{-n}\\
&=&t_{1,\frac{n-4}{2}}^{-1}t_{1,\frac{n-2}{2}}t_{1,n-1}(t_{1,\frac{n-6}{2}}^{-1}t_{1,\frac{n-4}{2}}t_{1,n-1}t_{1,\frac{n+6}{2}}\delta _0^{-n})\delta _0^{-n}\\
&=&t_{1,\frac{n-6}{2}}^{-1}t_{1,\frac{n-2}{2}}t_{1,n-1}^2\underline{t_{1,\frac{n+6}{2}}}\delta _0^{-2n}\\
&=&t_{1,\frac{n-6}{2}}^{-1}t_{1,\frac{n-2}{2}}t_{1,n-1}^2(t_{1,\frac{n-8}{2}}^{-1}t_{1,\frac{n-6}{2}}t_{1,n-1}t_{1,\frac{n+8}{2}}\delta _0^{-n})\delta _0^{-2n}\\
&=&t_{1,\frac{n-8}{2}}^{-1}t_{1,\frac{n-2}{2}}t_{1,n-1}^3\underline{t_{1,\frac{n+8}{2}}}\delta _0^{-3n}\\
&\vdots &\\
&=&t_{1,2}^{-1}t_{1,\frac{n-2}{2}}t_{1,n-1}^{\frac{n-6}{2}}\underline{t_{1,n-2}}\delta _0^{-\frac{n-6}{2}n}\\
&=&t_{1,2}^{-1}t_{1,\frac{n-2}{2}}t_{1,n-1}^{\frac{n-6}{2}}(t_{1,2}t_{1,n-1}^2\delta _0^{-n})\delta _0^{-\frac{n-6}{2}n}\\
&=&t_{1,\frac{n-2}{2}}t_{1,n-1}^{\frac{n-2}{2}}\delta _0^{-\frac{n-4}{2}n}. 
\end{eqnarray*}
Thus the relation~(\ref{rel4_3}) is equivalent to the relation
\begin{align*}
&t_{1,\frac{n-2}{2}}=t_{1,n-1}\underline{t_{1,\frac{n+2}{2}}}\delta _0^{-n}=t_{1,n-1}(t_{1,\frac{n-2}{2}}t_{1,n-1}^{\frac{n-2}{2}}\delta _0^{-\frac{n-4}{2}n})\delta _0^{-n}=t_{1,\frac{n-2}{2}}t_{1,n-1}^{\frac{n}{2}}\delta _0^{-\frac{n-2}{2}n}\\
&\Longleftrightarrow t_{1,n-1}^{\frac{n}{2}}\delta _0^{-\frac{n(n-2)}{2}}=1. 
\end{align*}
Therefore, this relation is equivalent to the relation~(2) in Lemma~\ref{lem_proof_abel2} and we have completed the proof of Lemma~\ref{lem_proof_abel2}.   
\end{proof}

\begin{proof}[The proof of Theorem~\ref{thm_abel}]
We will calculate $\HQ $ by using the finite presentation for $QB_n $ in Proposition~\ref{prop_pres_QB}. 
First, the relations~(1), (2), and (4) for $(i,j)=(1,n)$ in Proposition~\ref{prop_pres_QB} are trivial in $\HQ $. 
The relation~(3) in Proposition~\ref{prop_pres_QB} means that $t_{1,n}$ is a product of $\delta _0$.  
Remark that the relations $t_{i,j}=t_{1,j-i+1}$ for $1\leq i<j\leq n$ hold in $\HQ $ by Lemma~\ref{lem_proof_abel0} and the relations~(4) for $j<n$ in Proposition~\ref{prop_pres_QB} are obtained from these relations. 
The relations~(4) for $j=n$ and $i>1$ are equivalent to the relations in Lemmas~\ref{lem_proof_abel1} and~\ref{lem_proof_abel2}. 
The relations~(1) in Lemma~\ref{lem_proof_abel1} (resp. in Lemma~\ref{lem_proof_abel2}) imply that $t_{1,j}$ for $\frac{n+1}{2}\leq j\leq n-2$ (resp. for $\frac{n+2}{2}\leq j\leq n-2$) is a product of $t_{1,j}$ for $2\leq j\leq \frac{n-1}{2}$ (resp. for $2\leq j\leq \frac{n}{2}$), $t_{1,n-1}$, and $\delta _0$ in $\HQ $. 
Hence, when $n\geq 3$ is odd and as a presentation for an abelian group, we have
\begin{eqnarray*}
&&\HQ \\
&\cong &\left< t_{1,j}\ (2\leq j\leq \frac{n-1}{2}),\ t_{1,n-1},\ \delta _0 \middle|  t_{1,n-1}^n\delta _0^{-n(n-2)}=1\right> \\
&\cong &\left< t_{1,j}\ (2\leq j\leq \frac{n-1}{2}),\ t_{1,n-1},\ \delta _0,\ X \middle|  (t_{1,n-1}\delta _0^{-n+2})^n=1,\ X=t_{1,n-1}\delta _0^{-n+2}\right> \\
&\cong &\left< t_{1,j}\ (2\leq j\leq \frac{n-1}{2}),\ \delta _0,\ X \middle|  X^n=1\right> \\
&\cong &\Z ^{\frac{n-1}{2}}\oplus \Z _n.
\end{eqnarray*}

When $n\geq 4$ is even, we have
\begin{eqnarray*}
&&\HQ \\
&\cong &\left< t_{1,j}\ (2\leq j\leq \frac{n}{2}),\ t_{1,n-1},\ \delta _0 \middle|  t_{1,n-1}^{\frac{n}{2}}\delta _0^{-\frac{n(n-2)}{2}}=1\right> \\
&\cong &\left< t_{1,j}\ (2\leq j\leq \frac{n}{2}),\ t_{1,n-1},\ \delta _0,\ X \middle|  (t_{1,n-1}\delta _0^{-n+2})^{\frac{n}{2}}=1,\ X=t_{1,n-1}\delta _0^{-n+2}\right> \\
&\cong &\left< t_{1,j}\ (2\leq j\leq \frac{n}{2}),\ \delta _0,\ X \middle|  X^{\frac{n}{2}}=1\right> \\
&\cong &\Z ^{\frac{n}{2}}\oplus \Z _{\frac{n}{2}}.
\end{eqnarray*}
Therefore, we have completed the proof of Theorem~\ref{thm_abel}. 
\end{proof}

\section{Minimal generating sets for the quasitoric braid group}\label{section_mini-gen}

In this section, we prove Theorem~\ref{thm_mini-gen} and give two minimal generating sets for $QB_n$. 
Assume that $n\geq 3$. 
Recall that $\delta _0=\sigma _1\sigma _2\cdots \sigma _{n-1}\in QB_n$ and $t_{i,j}$ for $1\leq i<j\leq n$ is the full-twist braid from $i$-th to $j$-th strings as in Figure~\ref{fig_full-twist-def}. 
Theorem~\ref{thm_mini-gen} follows from the next theorem. 
\begin{thm}\label{thm_mini-gen1}
$QB_n$ is generated by $\delta _0$ and $t_{1,j}$ for $2\leq j\leq \frac{n+1}{2}$ if $n$ is odd, and for $2\leq j\leq \frac{n+2}{2}$ if $n$ is even. 
\end{thm} 

The generating set in Theorem~\ref{thm_mini-gen1} is a subset of the generating set of the presentation in Proposition~\ref{prop_pres_QB}. 

For $1\leq i\leq n-1$, we denote a $(n,1)$-quasitoric braid 
\[
\delta _i=\sigma _1\cdots \sigma _{n-i-1}\sigma _{n-i}^{-1}\cdots \sigma _{n-1}^{-1}\in QB_n.
\] 
The next theorem give another minimal generation get for $QB_n$. 

\begin{thm}\label{thm_mini-gen2}
$QB_n$ is generated by $\delta _i$ for $0\leq i\leq \frac{n-1}{2}$ if $n$ is odd, and for $0\leq j\leq \frac{n}{2}$ if $n$ is even. 
\end{thm} 

For conveniences, we put 
\[
N= \left\{ \begin{array}{ll}
 \frac{n+1}{2}&\text{if } n\geq 3 \text{ is odd},   \\
 \frac{n+2}{2}& \text{if } n\geq 4 \text{ is even}.
 \end{array} \right.
\]
We remark that $n-N\leq N-1$ since $n-N=\frac{n-1}{2}$ if $n$ is odd and $n-N=\frac{n-2}{2}$ if $n$ is even. 

\begin{proof}[Proof of Theorem~\ref{thm_mini-gen1}]
Let $G$ be a subgroup of $QB_n$ which is generated by $\delta _0$ and $t_{1,j}$ for $2\leq j\leq N$. 
By Proposition~\ref{prop_pres_QB}, $QB_n$ is generated by $\delta _0$ and $t_{i,j}$ for $1\leq i<j\leq n$. 
Since $t_{i,j}=\delta _0^{i-1}t_{1,j-i+1}\delta _0^{-(i-1)}$ for $2\leq i<j\leq n$ by Lemma~\ref{lem_t_ij-delta_conj}, we have $t_{i,j}\in G$ for $j-i+1\leq N$ and it is enough for completing the proof of Theorem~\ref{thm_mini-gen1} to prove that $t_{1,j}\in G$ for $N+1\leq j\leq n$. 

First, by the relation~(3) in Proposition~\ref{prop_pres_QB}, we have $t_{1,n}=\delta_0^n\in G$. 
When $n=3$ or $n=4$, since $n\geq j\geq N+1=n$, we have $t_{1,j}=t_{1,n}\in G$ and have completed the proof. 
We assume that $n\geq 5$ and recall that $t_{i,i}=1\in B_n$ for $1\leq i\leq n$.  
For $2\leq j\leq n$, by  lantern relation as in Figure~\ref{fig_lantern_rel-no2}, we have
\begin{eqnarray}
&& t_{1,n-1}t_{j+1,n}\cdot \delta _0^{-1}t_{1,j+1}\delta _0=t_{1,j}t_{j+1,n-1}t_{1,n}\notag \\ 
&\Longleftrightarrow & t_{1,n-1}=t_{1,j}t_{j+1,n-1}t_{1,n}\cdot \delta _0^{-1}t_{1,j+1}^{-1}\delta _0\cdot t_{j+1,n}^{-1}\label{eq_lantern_rel1}\\
&\Longleftrightarrow & t_{1,j}=t_{1,n-1}t_{j+1,n}\cdot \delta _0^{-1}t_{1,j+1}\delta _0\cdot t_{1,n}^{-1}t_{j+1,n-1}^{-1}.\label{eq_lantern_rel2}
\end{eqnarray}
By the relation~(\ref{eq_lantern_rel1}) for $j=N-1$, we have 
\[
t_{1,n-1}=t_{1,N-1}t_{N,n-1}t_{1,n}\cdot \delta _0^{-1}t_{1,N}^{-1}\delta _0\cdot t_{N,n}^{-1}.
\] 
Since $(n-1)-N+1=n-N\leq N-1\leq N$ and $n-N+1\leq N$, we have $t_{N,n-1},\ t_{N,n}\in G$. 
Hence we have $t_{1,n-1}\in G$.  

By the relations~(\ref{eq_lantern_rel2}) for $j=n-2, n-3, \dots , N+1$, we have inductively
\begin{eqnarray*}
t_{1,n-2}&=&t_{1,n-1}t_{n-1,n}\cdot \delta _0^{-1}t_{1,n-1}\delta _0\cdot t_{1,n}^{-1}t_{n-1,n-1}^{-1}\in G, \\
t_{1,n-3}&=&t_{1,n-1}t_{n-2,n}\cdot \delta _0^{-1}t_{1,n-2}\delta _0\cdot t_{1,n}^{-1}t_{n-2,n-1}^{-1}\in G,\\
t_{1,n-4}&=&t_{1,n-1}t_{n-3,n}\cdot \delta _0^{-1}t_{1,n-3}\delta _0\cdot t_{1,n}^{-1}t_{n-3,n-1}^{-1}\in G,\\
&\vdots &\\
t_{1,N+1}&=&t_{1,n-1}t_{N+2,n}\cdot \delta _0^{-1}t_{1,N+2}\delta _0\cdot t_{1,n}^{-1}t_{N+2,n-1}^{-1}\in G.
\end{eqnarray*}
Therefore $G=QB_n$ and we have completed the proof of Theorem~\ref{thm_mini-gen1}.  
\end{proof}

\begin{figure}[h]
\includegraphics[scale=0.8]{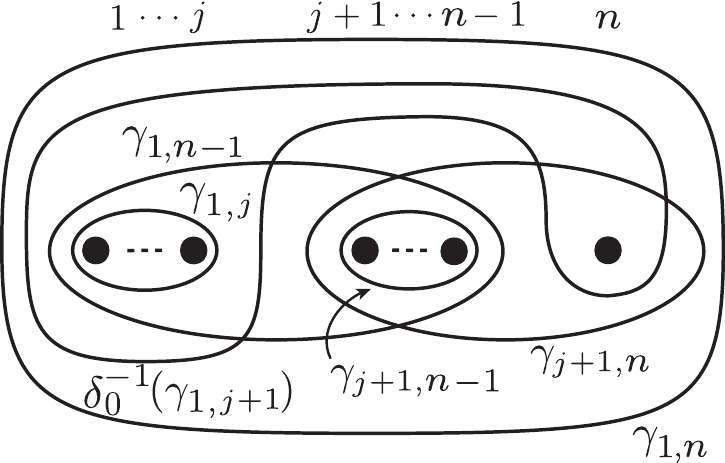}
\caption{Simple closed curves which are appeared in the lantern relation $t_{1,n-1}t_{j+1,n}\cdot \delta _0^{-1}t_{1,j+1}\delta _0=t_{1,j}t_{j+1,n-1}t_{1,n}$ for $2\leq j\leq n$.}\label{fig_lantern_rel-no2}
\end{figure}

\begin{proof}[Proof of Theorem~\ref{thm_mini-gen2}]
Let $G$ be a subgroup of $QB_n$ which is generated by $\delta _i$ for $0\leq i\leq N-1$. 
By Theorem~\ref{thm_mini-gen1}, $QB_n$ is generated by $\delta _0$ and $t_{1,j}$ for $2\leq j\leq N$. 
Hence, it is enough for completing the proof of Theorem~\ref{thm_mini-gen2} to prove that $t_{1,j}\in G$ for $2\leq j\leq N$. 
In particular, since $t_{1,j}=\delta _0^{-(n-j)}t_{n-j+1,n}\delta _0^{n-j}$ for $2\leq j\leq n$ by Lemma~\ref{lem_t_ij-delta_conj}, we prove $t_{n-j+1,n}\in G$ for $2\leq j\leq N$. 

First, for $N\leq i\leq n-1$, the element $\delta _{n-i}$ lies in $G$ since $n-i\leq n-N\leq N-1$. 
Hence, by Figure~\ref{fig_proof_smallgen1}, we have 
\begin{eqnarray}
\sigma _{n-1}^{-1}\cdots \sigma _i^{-1}\sigma _i^{-1}\cdots \sigma _{n-1}^{-1}=\delta _0^{-1}\delta _{n-i}\in G \label{eq_proof_mini1}
\end{eqnarray}
for $N\leq i\leq n-1$. 
Hence we have $t_{n-1,n}^{-1}=\sigma _{n-1}^{-2}=\delta _0^{-1}\delta _{1}\in G$. 
This implies that we have completed the proof of Theorem~\ref{thm_mini-gen2} when $n=3$. 

We assume that $n\geq 4$. 
Then we remark that $\delta _2\in G$ since $N-1\geq 2$. 
By Figure~\ref{fig_proof_smallgen2} and the relation~(\ref{eq_proof_mini1}), we have 
\[
t_{i,n}^{-1}=(\sigma _{n-1}^{-1}\cdots \sigma _i^{-1}\sigma _i^{-1}\cdots \sigma _{n-1}^{-1})t_{i,n-1}^{-1}=(\delta _0^{-1}\delta _{n-i})t_{i,n-1}^{-1}.
\] 
In particular, by using Lemma~\ref{lem_t_ij-delta_conj}, we have
\begin{eqnarray}
t_{i,n}^{-1}=(\delta _0^{-1}\delta _{n-i})\delta _0^{-1}t_{i+1,n}^{-1}\delta _0 \label{eq_proof_mini2}
\end{eqnarray}
for $1\leq i\leq n-1$. 
Hence, by using the relations~(\ref{eq_proof_mini1}) and~(\ref{eq_proof_mini2}), we show that $t_{n-j+1,n}^{-1}\in G$ for $3\leq j\leq N$ inductively as follows: 
\begin{eqnarray*}
t_{n-2,n}^{-1}&=&(\delta _0^{-1}\delta _{2})\delta _0^{-1}t_{n-1,n}^{-1}\delta _0 \in G,\\
t_{n-3,n}^{-1}&=&(\delta _0^{-1}\delta _{3})\delta _0^{-1}t_{n-2,n}^{-1}\delta _0 \in G,\\
&\vdots &\\
t_{n-N+1,n}^{-1}&=&(\delta _0^{-1}\delta _{N-1})\delta _0^{-1}t_{n-N+2,n}^{-1}\in G.
\end{eqnarray*}
Thus, we have $t_{n-j+1,n}\in G$ for $2\leq j\leq N$ and $G=QB_n$. 
Therefore, we have completed the proof of Theorem~\ref{thm_mini-gen2}.  
\end{proof}

\begin{figure}[h]
\includegraphics[scale=1.0]{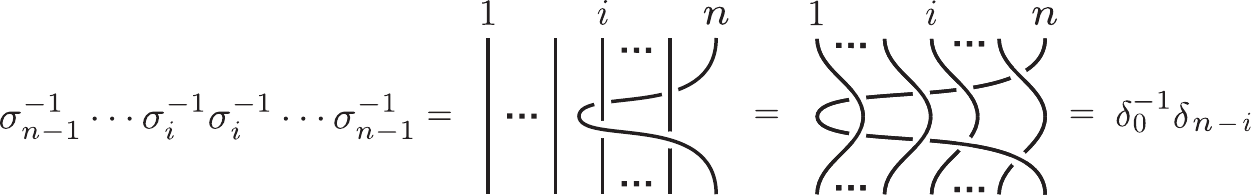}
\caption{The relation $\sigma _{n-1}^{-1}\cdots \sigma _i^{-1}\sigma _i^{-1}\cdots \sigma _{n-1}^{-1}=\delta _0^{-1}\delta _{n-i}$.}\label{fig_proof_smallgen1}
\end{figure}

\begin{figure}[h]
\includegraphics[scale=1.0]{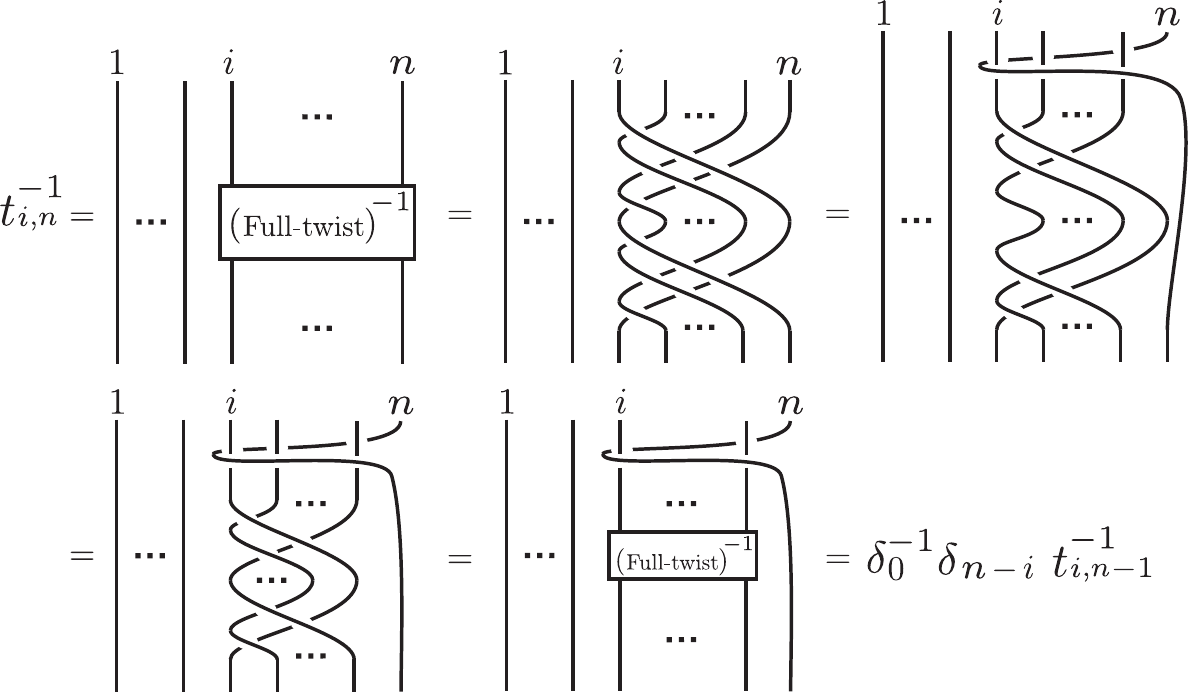}
\caption{The relation $t_{i,n}^{-1}=\delta _0^{-1}\delta _{n-1}t_{i,n-1}^{-1}$.}\label{fig_proof_smallgen2}
\end{figure}

\par
{\bf Acknowledgement:} The author would like to express his gratitude to Christoph Lamm, for helpful advices and telling him previous researches on rosette braids and relationship between rosette braids and quasitoric braids. 
The author was supported by JSPS KAKENHI Grant Number JP21K13794.


\begin{thebibliography}{99}

\bibitem{Alexander}
J. W. Alexander, \emph{A lemma on a system of knotted curves}, Proc. Nat. Acad. Sri. USA. \textbf{9} (1923), 93--95.







  





\bibitem{Hirose-Omori}
S. Hirose, G. Omori, \emph{Finite presentations for the balanced superelliptic mapping class groups}, arXiv:2203.13413. 


\bibitem{Johnson}
D. L. Johnson, \emph{Presentations of Groups}, London Math. Soc. Stud. Texts \textbf{15}, 1990. 


\bibitem{Lamm1}
C. Lamm, \emph{Zylinder-Knoten und symmetrische Vereinigungen}, Ph.D. thesis, University of Bonn, Bonner Mathematische Schriften 321 (1999). 

\bibitem{Lamm2}
C. Lamm, \emph{Fourier Knots}, arXiv:1210.4543 (English translation of a part of "Zylinder-Knoten und symmetrische Vereinigungen", Ph.D. thesis, University of Bonn, Bonner Mathematische Schriften 321 (1999)). 

\bibitem{Manturov}
V. O. Manturov, \emph{A combinatorial representation of links by quasitoric braids}, European J. Combin. \textbf{23} (2002), no. 2, 207--212.


\bibitem{Namanya}
C. Namanya, \emph{Pure braid group presentations via longest elements}, arXiv:2208.02120.


\bibitem{Omori1}
Genki Omori, \emph{A small generating set for the balanced superelliptic handlebody group}, to appear in Topology and its Applications.

\bibitem{Omori2}
G. Omori, \emph{The balanced superelliptic mapping class groups are generated by three elements}, arXiv:2203.14460. 


\bibitem{Shigeta}
T. Shigeta, \emph{Studies of knots via quasitoric braid presentations}, Tokyo University of Science, 2023, master's thesis, to appear (in Japanese). 




\bibitem{Vogel}
P. Vogel, \emph{Representation of links by braids: a new algorithm}, Comment. Math. Helv. \textbf{65} (1990), no. 1, 104--113. 

\end{thebibliography}
\end{document}